\documentclass[a4paper,11pt]{amsart}
\usepackage[utf8]{inputenc}
\usepackage{hyperref}
\usepackage[headings]{fullpage}
\usepackage{amsmath, amssymb}
\allowdisplaybreaks[1]
\usepackage[alphabetic,nobysame, initials]{amsrefs}

\newtheorem{thm}{Theorem} 
\newtheorem{lem}[thm]{Lemma}

\newtheorem{cor}[thm]{Corollary}
\theoremstyle{definition}
\newtheorem{defn}[thm]{Definition}

\newcommand{\Ne}{\mathbb N}
\newcommand{\NN}{\mathcal N}

\renewcommand{\epsilon}{\varepsilon}
\renewcommand{\phi}{\varphi}

\newcommand{\norm}[1]{\left\lVert#1\right\rVert}
\newcommand{\abs}[1]{\left\lvert#1\right\rvert}
\newcommand{\setbuilder}[2]{\{#1\colon#2\}}

\newcommand{\inter}{\operatorname{int}}

\newcommand{\un}[1]{\frac{1}{\norm{#1}}#1}
\newcommand{\card}[1]{\left\lvert#1\right\rvert}

\title{Sphere-of-influence graphs in normed spaces}
\dedicatory{Dedicated to K\'aroly Bezdek and Egon Schulte on the occasion of their 60th birthdays}
\author[M. Nasz\'odi, J. Pach, K. J. Swanepoel]{M\'arton Nasz\'odi \and J\'anos Pach \and 
Konrad Swanepoel}
\address{Department of Geometry, Lorand E\"otv\"os University, Pazm\'any P\'eter S\'etany 1/C Budapest, Hungary 1117}\email{marton.naszodi@math.elte.hu}
\address{EPFL Lausanne and R\'enyi Institute, Budapest}\email{pach@cims.nyu.edu}
\address{Department of Mathematics, London School of Economics and Political Science, Houghton Street, London WC2A 2AE, United Kingdom}\email{k.swanepoel@lse.ac.uk}
\subjclass[2010]{}
\thanks{M\'arton Nasz\'odi acknowledges the support of the J\'anos Bolyai Research Scholarship of the Hungarian Academy of Sciences, and
the Hung.\ Nat.\ Sci.\ Found.\ (OTKA) grant PD104744. Part of this paper was written when Swanepoel visited EPFL in April 2015. Research by J\'anos Pach was supported in part by Swiss National Science Foundation grants 200020-144531 and 200020-162884.}

\begin{document}
\begin{abstract}
We show that any $k$-th closed sphere-of-influence graph in a $d$-dimensional normed space has a vertex of degree less than~$5^d k$, thus obtaining a common generalization of results of F\"uredi and Loeb (1994) and Guibas, Pach and Sharir (1994).
\end{abstract}
\maketitle


Toussaint \cite{T88} introduced the sphere-of-influence graph of a finite set of points in Euclidean space for applications in pattern analysis and image processing (see \cite{T2014} for a recent survey).
This notion was later generalized to so-called closed sphere-of-influence graphs \cite{HJLM93} and to $k$-th closed sphere-of-influence graphs \cite{KZ2004}.
Our setting will be a $d$-dimensional normed space $\NN$ with norm $\norm{\cdot}$.
We denote the ball with center $c\in\NN$ and radius $r$ by $B(c,r)$.
\begin{defn}
Let $k\in\Ne$ and let $\setbuilder{c_i}{i=1,\dots,m}$ be a family of points in the 
$d$-dimensional normed space $\NN$.
For each $i=\{1,\dots,m\}$, let $r_i^{(k)}$ be the smallest $r$ such that 
\[\setbuilder{j\in\Ne}{j\neq i, \norm{c_i-c_j}\leq r}\] has at least $k$ elements.
Define the \emph{$k$-th closed sphere-of-influence graph on 
$V=\setbuilder{c_i}{i=1,\dots,m}$} by joining $c_i$ and $c_j$ whenever 
$B(c_i,r_i^{(k)})\cap B(c_j,r_j^{(k)})\neq\emptyset$.
\end{defn}
F\"uredi and Loeb \cite{FL94} gave an upper bound for the minimum degree of any closed sphere-of-influence graph in $\NN$ in terms of a certain packing quantity of the space (see also \cites{MQ94, S94}.)
\begin{defn}
Let $\vartheta(\NN)$ denote the largest number of points in the ball 
$B(o,2)$ of the normed space $\NN$ such that any two points are at distance 
at least $1$, and one of the points is the origin~$o$.
\end{defn}
F\"uredi and Loeb \cite{FL94} showed that any closed sphere-of-influence graph in $\NN$ has a vertex of degree smaller than $\vartheta(\NN)\leq 5^d$.
(It is clear that $\vartheta(\NN)$ is bounded above by the number of balls of radius $1/2$ that can be packed into a ball of radius $5/2$, which is at most $5^d$ by volume considerations.)

Guibas, Pach and Sharir \cite{GPS94} showed that any $k$-th closed sphere-of-influence graph in $d$-dimensional Euclidean space has a vertex of degree at most $c^dk$.
In this note we show the following more precise result, valid for all norms, and generalizing the result of F\"uredi and Loeb \cite{FL94} mentioned above.
\begin{thm}\label{thm:kcsig}
Every $k$-th sphere-of-influence graph on at least two points in a normed space $\NN$ has at least two vertices of degree smaller than
$\vartheta(\NN)k\leq 5^d k$.
\end{thm}
\begin{cor}
A $k$-th sphere-of-influence graph on $n$ points in $\NN$ has at most $(\vartheta(\NN)k-1)n\leq (5^d k -1)n$ edges.
\end{cor}
\begin{proof}[Proof of Theorem~\ref{thm:kcsig}]
Let $V=\{c_1,c_2,\dots,c_m\}$.
Relabel the vertices $c_1,c_2,\dots,c_m$ such that $r_1^{(k)}\leq r_2^{(k)}\leq\dots\leq r_m^{(k)}$.
We define an auxiliary graph $H$ on $V$ by joining $c_i$ and $c_j$ whenever 
$\norm{c_i-c_j}<\max\{r_i^{(k)},r_j^{(k)}\}$.
Thus, if $\setbuilder{c_i}{i\in I}$ is an independent set in $H$, then no ball in $\setbuilder{B(c_i,r_i^{(k)})}{i\in I}$ contains the center of another in its interior.
We next bound the chromatic number of $H$.
\begin{lem}\label{lem:color}
The chromatic number of $H$ does not exceed $k$.
\end{lem}
\begin{proof}
Note that for each $i\in\{1,\dots,m\}$, the set \[\setbuilder{j<i}{c_ic_j\in E(H)} = 
\setbuilder{j<i}{\norm{c_i-c_j}<r_i^{(k)}}\] has less than $k$ elements.
Therefore, we can greedily color $H$ in the order $c_1, 
c_2,\dots, c_m$ by $k$ colors.
\end{proof}
We next show that the degrees of $c_1$ and $c_2$ (corresponding to the two smallest $r_i^{(k)}$) are both at most $\vartheta(\NN)k$, which will complete the proof of Theorem~\ref{thm:kcsig}.
We first need the so-called ``bow-and-arrow'' inequality of \cite{FL94}.
\begin{lem}[F\"uredi--Loeb \cite{FL94}]\label{cor:bowarrow}
For any two non-zero elements $a$ and $b$ of a normed space,
\[ \norm{\un{a}-\un{b}}\geq\frac{\norm{a-b}-\abs{\norm{a}-\norm{b}}}{\norm{b}}.\]
\end{lem}

\begin{proof}
Without loss of generality, $\norm{a}\geq\norm{b}>0$.
Then
\begin{align*}
\norm{a-b} &=\norm{\norm{a}\un{a}-\norm{b}\un{b}}\\
&=\norm{\norm{b}(\un{a}-\un{b})+(\norm{a}-\norm{b})\un{a}}\\
&\leq \norm{b}\norm{\un{a}-\un{b}}+\norm{a}-\norm{b}. \qedhere
\end{align*}
\end{proof}

The next lemma is abstracted with minimal hypotheses from \cite{MQ94}*{Proof of Theorem~6} (see also 
\cite{FL94}*{Proof of Theorem~2.1}).
\begin{lem}\label{lem:satellite}
Consider the balls $B(v_1,\lambda_1)$ and $B(v_2,\lambda_2)$ in the normed space $\NN$, such that
$\max\{\lambda_1,\lambda_2\}\geq 1$, $v_1\notin \inter(B(v_2,\lambda_2))$, $v_2\notin \inter(B(v_1,\lambda_1))$ and $B(v_i,\lambda_i)\cap B(o,1)\neq\emptyset$ \textup{(}$i=1,2$\textup{)}.
Define $\pi\colon\NN\to B(o,2)$ by \[\pi(x)=\begin{cases} x &\text{if $\norm{x}\leq 2$,}\\ \frac{2}{\norm{x}}x & \text{if $\norm{x}\geq 2$.} \end{cases}\]
Then $\norm{\pi(v_1)-\pi(v_2)}\geq 1$.
\end{lem}
\begin{proof}
In terms of the norm, we are given that $\norm{v_1-v_2}\geq\max\{\lambda_1,\lambda_2\}\geq 1$, $\norm{v_1}\leq\lambda_1+1$, and $\norm{v_2}\leq\lambda_2+1$.
Without loss of generality, $\norm{v_2}\leq\norm{v_1}$.

If $v_1,v_2\in 2K$ then $\norm{\pi(v_1)-\pi(v_2)}=\norm{v_1-v_2}\geq 1$.

If $v_1\notin 2K$ and $v_2\in 2K$, then
\begin{align*}
\norm{\pi(v_1)-\pi(v_2)} &=\norm{2\un{v_1}-v_2} \geq \norm{v_1-v_2}-\norm{v_1-2\un{v_1}}\\
&= \norm{v_1-v_2}-(\norm{v_1}-2)\geq\lambda_1 - (\lambda_1+1) + 2 =1.
\end{align*}
If $v_1,v_2\notin 2K$, then
\begin{align*}
\norm{\pi(v_1)-\pi(v_2)} &=\norm{2\un{v_1}-2\un{v_2}} \geq 2\frac{\norm{v_1-v_2}-\norm{v_1}+\norm{v_2}}{\norm{v_2}} \quad\text{by  Lemma~\ref{cor:bowarrow}}\\
&\geq 2\left(\frac{\lambda_1-(\lambda_1 + 1)}{\norm{v_2}} + 1\right) = \frac{-2}{\norm{v_2}} + 2 \geq -1+2=1.\qedhere
\end{align*}
\end{proof}
We can now finish the proof of Theorem~\ref{thm:kcsig}.
Let $c\in\{c_1,c_2\}$ be the point with smallest or second-smallest $r_i^{(k)}$.
By Lemma~\ref{lem:color} we can partition the set of neighbors of $c$ in the $k$-th closed sphere-of-influence graph on $V$ into $k$ classes $N_1,\dots,N_k$ so that each $N_i$ is an independent set in $H$.
We may assume that the radius $r_i^{(k)}$ corresponding to $c$ is 1.
Then each ball in $\setbuilder{B(c_j,r_j^{(k)})}{c_j\in N_i}$ intersects $B(c,1)$, and the center of no ball is in the interior of another ball.
By Lemma~\ref{lem:satellite}, $\setbuilder{\pi(p-c)}{p\in N_i}$ is a set of points contained in $B(o,2)$ with a distance of at least $1$ between any two.
That is, $\card{N_i\setminus\inter(B(c,1))} \leq\vartheta(\NN)-1$ for each $i=1,\dots,k$.
Since there are at most $k-1$ points in $V\cap \inter(B(c,1))\setminus\{c\}$, 
it follows that the degree of $c$ is at most $\sum_{i=1}^k \card{N_i\setminus\inter(B(c,1))} + k-1 \leq(\vartheta(\NN)-1)k + k-1 = \vartheta(\NN)k-1$.
\end{proof}

\bibliographystyle{amsalpha}
\bibliography{biblio}

\end{document}